\documentclass{amsart}
\usepackage{amsmath, amssymb,epic,graphicx,mathrsfs,enumerate}
\usepackage[all]{xy}
\usepackage{color}
\usepackage{comment}

\usepackage{amsthm}
\usepackage{amssymb}
\usepackage{latexsym}
\usepackage{longtable}
\usepackage{epsfig}
\usepackage{amsmath}
\usepackage{hhline}

\newtheorem{thm}{Theorem}
\newtheorem{cor}[thm]{Corollary}

\newtheorem{question}[]{Question}

\numberwithin{equation}{section}

\renewcommand{\footnote}{\endnote}
\newcommand{\ignore}[1]{}\makeglossary

\begin{document}
	\bibliographystyle{amsplain}
	\title[K\"{o}v\'{a}ri-S\'{o}s-Tur\'{a}n theorem]{Applying the K\"{o}v\'{a}ri-S\'{o}s-Tur\'{a}n theorem\\ to a question in group theory}

	\author{Andrea Lucchini}
	\address{Andrea Lucchini\\ Universit\`a degli Studi di Padova\\  Dipartimento di Matematica \lq\lq Tullio Levi-Civita\rq\rq\\ Via Trieste 63, 35121 Padova, Italy\\email: lucchini@math.unipd.it}

	\begin{abstract}Let $m\leq n$ be positive integers and $\mathfrak X$  a class of groups which is closed for
		subgroups, quotient groups and extensions. Suppose that a finite group $G$  satisfies the condition that for every two subsets $M$
		and $N$ of cardinalities $m$ and $n,$ respectively, there exist $x \in M$ and $y \in N$ such that $\langle x, y \rangle\in \mathfrak X.$ Then either $G\in \mathfrak X$
		or  $|G|\leq \left(\frac{180}{53}\right)^m(n-1).$	\end{abstract}
	\maketitle

Let $m, n$ be positive integers and $\mathfrak X$ be a class of groups.
We say that a group $G$ satisfies the condition $\mathfrak X(m,n)$ if for every two subsets $M$
and $N$ of cardinalities $m$ and $n,$ respectively, there exist $x \in M$ and $y \in N$ such that $\langle x, y \rangle\in \mathfrak X.$
If $G$ satisfies the condition $\mathfrak X(m,n)$, then we write $G \in \mathfrak X(m,n).$ In \cite{zorro} M. Zarrin proposed the following question.

\begin{question} Let G be a finite group and $G\notin \mathfrak X.$  Does there exist a bound
(depending only on $m$ and $n$) for the size of $G$ if $G$ satisfies the
condition $\mathfrak X(m,n)$?
\end{question}

An affirmative answer is given in \cite{zorro} for the class of nilpotent groups. In an earlier paper R. Bryce gave a positive solution for the class of supersoluble groups, under the additional condition $n=m.$ In this short note we prove that an affirmative question can be given whenever $\mathfrak X$ is a class of finite groups which is closed for
subgroups, quotient groups and extensions. Our argument relies on
the	K\"{o}v\'{a}ri-S\'{o}s-Tur\'{a}n theorem \cite{kst}, stating that, if $m\leq n$ are two positive integers, then a graph with $t$ vertices and at least $((n-1)^{1/m}t^{2-1/m}+(m-1)t)/2$ edges, contains a copy of the complete bipartite graph $K_{m,n}.$ The crucial observation is the following:
\begin{thm}\label{prin}
Let $\mathfrak X$ be a class of groups and suppose that there exists
a real positive number $\gamma$ with the following property: if $X$ is a finite group and the probability that two randomly chosen elements of $X$ generate a group in  $\mathfrak X$ is
greater than $\gamma,$ then $X$ is in  $\mathfrak X$. If $m\leq n,$ then $$|G|\leq \left(\frac{2}{1-\gamma}\right)^m(n-1)$$ for any $G\in  \mathfrak X(m,n)\setminus \mathfrak X.$
\end{thm}

\begin{proof}
	Let $G\in \mathfrak X(m,n)\setminus \mathfrak X.$ Consider the graph $\Gamma_{\mathfrak X}(G)$ whose vertices are the elements of $G$ and in which two vertices $x_1$ and $x_2$ are  joined by an edge if and only if $\langle x_1, x_2 \rangle \notin \mathfrak X$ and let $\eta$ the number of edges of $\Gamma_{\mathfrak X}(G).$ Since $G\notin \mathfrak X,$ the probability that two vertices of $\Gamma_{\mathfrak X}(G)$ are joined by an edge is at least $1-\gamma$, so we must have
	\begin{equation}\label{uno}
	\eta \geq \frac{(1-\gamma)|G|^2}{2}.
	\end{equation}
	On the other hand, since $G\in   \mathfrak X(m,n),$  $\Gamma_{\mathfrak X}(G)$ cannot contain the complete bipartite graph $K_{m,n}$ as a subgraph.
 By the
	K\"{o}v\'{a}ri-S\'{o}s-Tur\'{a}n theorem, 
	\begin{equation}\label{due}\eta\leq \frac{(n-1)^{1/m}|G|^{2-1/m}+(m-1)|G|}2.
	\end{equation}
	Combining (\ref{uno}) and \ref{due}, we deduce
	\begin{equation}\label{tre}
	\left( \frac{n-1}{|G|}\right)^{1/m}+\frac{n-1}{|G|}\geq 
		\left( \frac{n-1}{|G|}\right)^{1/m}+\frac{m-1}{|G|}\geq 1-\gamma.
	\end{equation}
We may assume $|G|\geq n-1$. This implies 	$\left(\frac{n-1}{|G|}\right)^{1/m}\geq\frac{n-1}{|G|}$ and therefore it follows from (\ref{tre}) that
\begin{equation}
\left(\frac{n-1}{|G|}\right)^{1/m}\geq \frac{1-\gamma}2.
\end{equation}
This implies
$$|G|\leq \left(\frac{2}{1-\gamma}\right)^m(n-1).\qedhere$$
\end{proof}
Guralnick and Wilson \cite{gw}, using the classification of the finite simple groups, proved the following result. There exists a real number $\kappa$, strictly between 0 and 1, with the
following property: let $\mathfrak X$ be any class of finite groups which is closed for
subgroups, quotient groups and extensions, and let $G$ be a finite group; if the
probability that two randomly chosen elements of $G$ generate a group in $\mathfrak X$ is
greater than $\kappa$, then $G$ is in $\mathfrak X.$ 
Combining \cite[Proposition 5]{gw} with \cite[Theorem 1.1]{nina}, one may deduce that $\kappa$ can be taken to be $\frac{37}{90}=\max\left(1-\frac{53}{90},\frac{5}{18}\right).$
This allows us to deduce our main result.
\begin{cor}\label{ma}
Let $\mathfrak X$ be any class of finite groups which is closed for
subgroups, quotient groups and extensions, and let $G$ be a finite group. If $m\leq n$ are positive integers and $G\in  \mathfrak X(m,n)\setminus \mathfrak X,$ then $|G|\leq \left(\frac{180}{53}\right)^m(n-1)$.
\end{cor}
With the same argument, combining Theorem \ref{prin} with \cite[Theorem A]{gw} (see also the remark in \cite{gw} following the statement of Theorem A), we deduce the following results, the first of which is an improvement of  \cite[Theorem 3.6]{zorro}.
\begin{cor}
Let $m\leq n$ be positive integers and  $G$  a finite group.
\begin{enumerate}
	\item 
 If $\mathfrak X$ is the class of nilpotent groups and
	$G\in  \mathfrak X(m,n)\setminus \mathfrak X,$ then $$|G|\leq 4^m(n-1).$$
	\item  If $\mathfrak X$ is the class of soluble groups and
	$G\in  \mathfrak X(m,n)\setminus \mathfrak X,$ then $$|G|\leq \left(\frac{60}{19}\right)^m(n-1).$$
		\item  If $\mathfrak X$ is the class of finite groups of odd order and
	$G\in  \mathfrak X(m,n)\setminus \mathfrak X,$ then $$|G|\leq \left(\frac{8}{3}\right)^m(n-1).$$
\end{enumerate}
\end{cor}

\end{document}